\documentclass{proc-l}

\usepackage{amssymb}
\usepackage{slashed}

\newtheorem{theorem}{Theorem}[section]
\newtheorem{lemma}[theorem]{Lemma}

\newtheorem{corollary}[theorem]{Corollary}

\theoremstyle{definition}
\newtheorem{definition}[theorem]{Definition}
\newtheorem{example}[theorem]{Example}

\theoremstyle{remark}
\newtheorem{remark}[theorem]{Remark}

\numberwithin{equation}{section}

\newcommand{\ZZ}{\mathbb{Z}}
\newcommand{\NN}{\mathbb{N}}
\newcommand{\RR}{\mathbb{R}}
\newcommand{\C}{\mathbb{C}}
\newcommand{\rationals}{\mathbb{Q}}

\newcommand{\calS}{\mathcal{S}}
\newcommand{\calR}{\mathcal{R}}
\newcommand{\tensor}{\otimes}

\DeclareMathOperator{\scal}{scal}
\DeclareMathOperator{\Length}{Length}
\DeclareMathOperator{\Area}{Area}
\DeclareMathOperator{\Spin}{Spin}
\DeclareMathOperator{\ind}{ind}

\DeclareMathOperator{\Ker}{Ker}
\DeclareMathOperator{\End}{End}    

\begin{document}

\title{Enlargeable metrics on nonspin manifolds\protect}


%    author one information
\author{Simone Cecchini}
\address{Mathematisches Institut,
Georg-August-Universit\"at, 
G{\"o}ttingen,
Germany}
\email{cecchini@mathematik.uni-goettingen.de}
%\thanks{}

%    author two information
\author{Thomas Schick}
\address{Mathematisches Institut,
Georg-August-Universit\"at, 
G{\"o}ttingen,
Germany}
\email{thomas.schick@math.uni-goettingen.de}
\thanks{Both authors thank the DFG SPP 2026 for support}

%    \subjclass is required.

\date{}

%\dedicatory{}

%    "Communicated by" -- provide editor's name; required.
\commby{Jiaping Wang}

%    Abstract is required.
\begin{abstract}
We show that an enlargeable Riemannian metric on a (possibly nonspin) manifold cannot have uniformly positive scalar curvature.
This extends a well-known result of Gromov and Lawson to the nonspin setting.
We also prove that every noncompact manifold admits a nonenlargeable metric.
In proving the first result, we use the main result of the recent paper by
Schoen and Yau on minimal hypersurfaces to obstruct positive
scalar curvature in arbitrary dimensions.
More concretely, we use this  to study nonzero degree maps $f\colon X^n\rightarrow S^k\times T^{n-k}$, with $k=1,2,3$.
When $X$ is a closed oriented manifold endowed with a metric $g$ of positive
scalar curvature and the map $f$ is (possibly area) contracting, we prove
inequalities relating the lower bound of the scalar curvature of $g$ and the
contracting factor of the map $f$.
\end{abstract}

\maketitle

\section{Introduction}
It has been an important topic in differential geometry in recent decades to construct obstructions to the existence of metrics of positive scalar curvature on a smooth manifold.
There are two main methods for this.
The first method is due to Lichnerowitz~\cite{Li63} and Atiyah-Singer~\cite{ASIII} and makes use of the index theory of the spin Dirac operator.
The main restriction of this method is that it applies only to spin manifolds, or at least to manifolds with a spin cover. 
The second method is due to Schoen and Yau~\cite{SY79} and is based on the spectral properties of the conformal Laplacian on a stable minimal hypersurfaces of a closed oriented manifold $M$.
It implies that if $M$ carries a metric of positive scalar curvature, and if $3\leq \dim M\leq 8$, then every nonzero homology class $\alpha\in H_{n-1}(M;\ZZ)$ is represented by a smooth embedded hypersurface $N\subset M$ carrying a metric of positive scalar curvature.
The main limitation of this technique is that it requires the dimension of $M$
to be at most $7$ in its original incarnation, and $8$ by the observation made
in \cite{JoachimSchick}, using \cite{Smale}.
In a recent work~\cite{SY17}, Schoen and Yau were able to remove this dimensional restriction at least in certain situations.
This suggests to ask whether obstructions constructed by using the spin Dirac
operator can be extended to the nonspin case by using their minimal
$k$-slicing technique.

In this paper we focus on the notion of enlargeability, that in the spin case
has been proved to be an obstruction to positive scalar curvature by Gromov
and Lawson~\cite{GL83}, compare also \cite{HS1,HS2}.
We start with recalling some definitions.
Let $f\colon (X,g)\rightarrow (Y,g_Y)$ be a differentiable map between smooth Riemannian manifolds.
We say that $f$ is \emph{$k$-dimensionally $\epsilon$-contracting} if for each $x\in X$ 
\[
	\|f_\ast(v_1)\wedge\cdots\wedge f_\ast(v_k)\|_{f(x)}\ \leq\ \epsilon\, \|v_1\wedge\cdots\wedge v_k\|_x\,,
	\qquad\qquad \forall v_1,\ldots,v_k\in T_xX\,.
\]
When $k=1$, we say that $f$ is \emph{$\epsilon$-contracting}.
When $k=2$, $f$ is called \emph{area $\epsilon$-contracting}.
Notice that an $\epsilon$-contracting map is area $\epsilon^2$-contracting.
Notice also that an $\epsilon$-contracting map contracts lengths by an $\epsilon$-factor and that an area $\epsilon$-contracting map contracts areas by an $\epsilon$-factor.

%-------------------------------------------------------------------------------------------------------------------------------------------------------------
%-------------------------------------------------------------------------------------------------------------------------------------------------------------

\begin{definition}
Let $M$ be a smooth oriented connected manifold of dimension $n$ and let $k\in\NN$. 
We say that a Riemannian metric $g$ on $M$ is enlargeable in dimension $k$ if
for every $\epsilon>0$ there exist a connected covering $\bar M\rightarrow M$
and a map $f\colon (\bar M,\bar g)\rightarrow (S^n,ds_n^2)$ such that 
\begin{itemize}
  \item $f$ is compactly supported, i.e.~there is a compact subset $K\subset
     \bar M$ such that $f|_{\bar M\setminus K}$ is constant;
  \item $f$ has nonzero degree (making sense as $f$ is compactly supported);
  \item $f$ is $k$-dimensionally $\epsilon$-contracting.
\end{itemize}
Here, $\bar g$ is the lift of the metric $g$ to $\bar M$ and $(S^n,ds_n^2)$ is the $n$-dimensional sphere endowed with the canonical round metric.
The manifold $M$ is called enlargeable in dimension $k$ if every Riemannian metric $g$ on $M$ (complete or not) is enlargeable in dimension $k$.
If $k=2$ the concept is called ``area-enlargeability'', if $k=1$ we simply talk of ``enlargeability''. 
\end{definition}

%-------------------------------------------------------------------------------------------------------------------------------------------------------------

\begin{remark}
Our definition of enlargeability is the one that is adopted
in~\cite{LM89,HS1,HS2}. 
In contrast, in~\cite{GL80} and~\cite{GL83} it is required that the covers $\bar M$ admit a spin structure.
\end{remark}

%-------------------------------------------------------------------------------------------------------------------------------------------------------------

\begin{remark}
Any enlargeable manifold is also area-enlargeable, but the converse is not
necessarily  true: see Remark~\ref{R:area-enlargeable but not enlargeable}.
\end{remark}

%-------------------------------------------------------------------------------------------------------------------------------------------------------------

\begin{remark}
There are many examples of closed enlargeable manifolds.
For instance, the $n$-dimensional torus, compact ``solvmanifolds" and closed manifolds of nonpositive sectional curvature are all enlargeable.
Note that, in this case, it suffices to check the conditions for a single metric, as all metrics are equivalent.
The notion of area-enlargeability is particulary interesting in the noncompact case.
For example, if $M$ is a closed enlargeable manifold, then $M\times \RR$ is area-enlargeable.
For a comprehensive discussion, we refer the reader to~\cite{GL80} and~\cite{GL83}.
\end{remark}

%-------------------------------------------------------------------------------------------------------------------------------------------------------------

Gromov and Lawson proved by using the spin Dirac operator technique that area-enlargeability (and \emph{a fortiori} enlargeability) is an obstruction for a metric to have uniformly positive scalar curvature.

%-------------------------------------------------------------------------------------------------------------------------------------------------------------
%-------------------------------------------------------------------------------------------------------------------------------------------------------------

\begin{theorem}[Gromov-Lawson]\label{T:GL1}
Let $X$ be a spin manifold without boundary.
A complete area-enlargeable metric $g$ on $X$ cannot have uniformly positive scalar curvature.
\end{theorem}

%-------------------------------------------------------------------------------------------------------------------------------------------------------------

\noindent
This theorem directly implies the following consequence.

%-------------------------------------------------------------------------------------------------------------------------------------------------------------
%-------------------------------------------------------------------------------------------------------------------------------------------------------------

\begin{theorem}[Gromov-Lawson]\label{T:GL2}
A closed enlargeable spin manifold cannot carry any metric of positive scalar curvature.
\end{theorem}

%-------------------------------------------------------------------------------------------------------------------------------------------------------------

\noindent
In the case of area-enlargeable manifolds, Gromov and Lawson used Theorem~\ref{T:GL1} to deduce the following consequence.

\begin{theorem}[Gromov-Lawson]\label{T:GL3}
An area-enlargeable spin manifold cannot carry any complete metric of positive scalar curvature.
\end{theorem}

%-------------------------------------------------------------------------------------------------------------------------------------------------------------

Our first result extends, in the case of enlargeable metrics, Theorem~\ref{T:GL1} to nonspin manifolds.
This immediately generalizes Theorem~\ref{T:GL2} to the nonspin setting.

%-------------------------------------------------------------------------------------------------------------------------------------------------------------
%-------------------------------------------------------------------------------------------------------------------------------------------------------------

\begin{theorem}\label{T:enlargeable nonspin}
An enlargeable metric (complete or not) cannot have uniformly positive scalar curvature.
Therefore, a closed enlargeable manifold cannot carry any metric of positive scalar curvature.
\end{theorem}

%-------------------------------------------------------------------------------------------------------------------------------------------------------------

\begin{remark}
This theorem extends Theorem~\ref{T:GL1}, in the case of enlargeable metrics, in two directions.
In fact, we drop both the spin condition on the manifold and the completeness assumption on the metric.
\end{remark}

%-------------------------------------------------------------------------------------------------------------------------------------------------------------

\begin{remark}
In the case of area-enlargeable metrics, the completeness assumption on the metric in Theorem~\ref{T:GL1} cannot be dropped, as the next example shows.
\end{remark}

%-------------------------------------------------------------------------------------------------------------------------------------------------------------

\begin{example}
Let $X$ be any closed manifold.
Using~\cite{Gr69}, endow $X\times (-2,2)$ with a (possibly incomplete) metric of positive sectional curvature.
The inclusion $X\times (-1,1)\subset X\times (-2,2)$ induces on $X\times
(-1,1)$ an incomplete metric of uniformly positive scalar (even sectional) curvature.
Use the diffeomorphism $X\times (-1,1)\cong X\times\RR$ to endow $X\times\RR$ with a metric $g$ of uniformly positive scalar curvauture.
When $X$ is enlargeable, $g$ is area-enlargeable.
Thus, in this case we obtain an example of an incomplete area-enlargeable metric of uniformly positive scalar curvature.
\end{example}

%-------------------------------------------------------------------------------------------------------------------------------------------------------------

As we already noted, there are many examples of compact enlargeable manifolds.
On the other hand, there are no known examples of noncompact enlargeable manifolds.
It is natural to ask whether such objects exist at all.
Our second result gives a negative answer to this question.

%-------------------------------------------------------------------------------------------------------------------------------------------------------------
%-------------------------------------------------------------------------------------------------------------------------------------------------------------

\begin{theorem}\label{T:nonenlargeable metrics}
A noncompact manifold cannot be enlargeable.
\end{theorem}

%-------------------------------------------------------------------------------------------------------------------------------------------------------------

\begin{remark}\label{R:area-enlargeable but not enlargeable}
This theorem provides many examples of manifolds that are area-enlargeable, but not enlargeable.
In particular, any noncompact area-enlargeable manifold is of this type.
So far, there are no known examples of closed manifolds that are area-enlargeable but not enlargeable.
It would be interesting to know whether such examples exist.
We leave this question as a challenge for the reader.
\end{remark}

%-------------------------------------------------------------------------------------------------------------------------------------------------------------

In order to prove Theorem~\ref{T:enlargeable nonspin}, we make use of the
results recently obtained by Schoen and Yau~\cite{SY17}.
We study maps $f\colon (X,g)\rightarrow \left(S^k\times T^{n-k},ds_k^2+dt_{n-k}^2\right)$, where $(X,g)$ is a closed $n$-dimensional oriented Riemannian manifold, $ds_k^2$ is the standard round metric on the sphere $S^k$, and $dt_{n-k}^2$ is the standard flat metric on the torus $T^{n-k}$.
If the map $f$ has nonzero degree, then $f$ induces a minimal $k$-slicing of $X$ in the sense of~\cite{SY17}.
This means that there exists a nested family of (singular) hypersurfaces $\Sigma_k\subset\cdots\subset \Sigma_n=X$, where each $\Sigma_j$ is a minimizer in $\Sigma_{j+1}$ with respect to a suitable weighted volume.
Assuming that the scalar curvature of $g$ is positive and the map $f$ is
(possibly area) contracting, using the properties of $k$-slicings derived
in~\cite{SY17} we obtain inequalities relating the contracting factor of $f$
and the lower bound of the scalar curvature of $g$.

%-------------------------------------------------------------------------------------------------------------------------------------------------------------
%-------------------------------------------------------------------------------------------------------------------------------------------------------------

\begin{theorem}\label{T:curvature inequality}
Let $(X,g)$ be a closed oriented $n$-dimensional Riemannian manifold, let
$f\colon (X,g)\rightarrow \left(S^k\times T^{n-k},ds_k^2+dt_{n-k}^2\right)$ be a smooth map of nonzero degree, and let $k_0$ be a positive number.
Then:
\begin{enumerate}
\item Suppose $k=1$.
	If there exists a connected open subset $J\subset S^1$ such that $\scal(g)\geq k_0$ on $f^{-1}(J\times T^{n-1})$ 
	and $f$ is $\epsilon$-contracting on $f^{-1}(J\times T^{n-1})$, 
	then $2\pi\,\epsilon~\geq~\sqrt{k_0}\,\Length(J)$.
\item Suppose $k=2$ and $f$ is area $\epsilon$-contracting. 
	If $\scal(g)\geq k_0$, then $4\,\epsilon \geq k_0$.
\item Suppose $k=3$, $n\leq 8$, and $f$ is area $\epsilon$-contracting.
	If $\scal(g)\geq k_0$, then $6\,\epsilon \geq k_0$.
\end{enumerate}
\end{theorem}

%-------------------------------------------------------------------------------------------------------------------------------------------------------------
%-------------------------------------------------------------------------------------------------------------------------------------------------------------

\begin{remark}
The inequality in Part~(a) of Theorem~\ref{T:curvature inequality} only requires control of the metric on a region of $X$.
In Section~\ref{S:cut-and-paste}, we use this fact and a ``cut-and-paste" construction to prove Theorem~\ref{T:enlargeable nonspin}.
\end{remark}

%-------------------------------------------------------------------------------------------------------------------------------------------------------------

\begin{remark}
The inequalities in Parts~(b) and~(c) of Theorem~\ref{T:curvature inequality} on the other hand require a global control of the metric $g$.
This is the main difficulty in using them to extend Theorem~\ref{T:GL1} to the nonspin setting in the case of area-enlargeable metrics.
Notice that this would allow us to drop the spin condition from Theorem~\ref{T:GL3}.
\end{remark}

%-------------------------------------------------------------------------------------------------------------------------------------------------------------

\begin{remark}
The dimension assumption in Part~(c) is due the fact that the techniques in~\cite{SY17}, applied to maps $X^n\rightarrow S^3\times T^{n-3}$, produce singularities when $n>8$.
We plan to treat this case in a future paper.
\end{remark}

%-------------------------------------------------------------------------------------------------------------------------------------------------------------

\begin{remark}
When $\dim X\leq 8$, Theorem~\ref{T:curvature inequality} can be proved using
the classical results of Schoen and Yau~\cite{SY79}. 
The same observation applies to Theorem~\ref{T:enlargeable nonspin}.
The new techniques of Schoen and Yau are needed for the case of manifolds of dimension greater than $8$.
\end{remark}

%-------------------------------------------------------------------------------------------------------------------------------------------------------------

\begin{remark}
In~\cite[Section~12]{GL83} Gromov and Lawson used the minimal hypersurface technique to drop the spin assumption from Theorem~\ref{T:GL2}  for manifolds of dimension $\leq 7$.
Our construction differs from theirs and, as we believe, is more
transparent. For these dimensions, our result gives
alternative point of view.
The possibility of using the new results of Schoen and Yau to drop the spin 
assumption from Theorem~\ref{T:GL2} in any dimension is suggested by Gromov
in~\cite{Gr18}. Our proof is designed in such a way as to only use the most
standard constructions from geometric measure theory: no symmetrization, no
manifolds with boundary. That way, we can use the new results of Schoen and
Yau \cite{SY17} ``out of the box'' without any need to refine and generalize
them. 
\end{remark}

%-------------------------------------------------------------------------------------------------------------------------------------------------------------

The paper is organized as follows.
In Section~\ref{S:k-slicings} we present the notion and results of minimal
$k$-slicing recently introduced by Schoen and Yau and use it to deduce
Theorem~\ref{T:curvature inequality}.
In Section~\ref{S:cut-and-paste} we present a ``cut-and-paste" construction that allows us to deduce Theorem~\ref{T:enlargeable nonspin} from Part~(a) of Theorem~\ref{T:curvature inequality}.
Finally, in Section~\ref{S:nonenlargeable metrics} we show that every noncompact manifold carries a nonenlargeable metric.
This proves Theorem~\ref{T:nonenlargeable metrics}.

%-------------------------------------------------------------------------------------------------------------------------------------------------------------
%-------------------------------------------------------------------------------------------------------------------------------------------------------------
%-------------------------------------------------------------------------------------------------------------------------------------------------------------
%-------------------------------------------------------------------------------------------------------------------------------------------------------------
%-------------------------------------------------------------------------------------------------------------------------------------------------------------
%-------------------------------------------------------------------------------------------------------------------------------------------------------------
%-------------------------------------------------------------------------------------------------------------------------------------------------------------
%-------------------------------------------------------------------------------------------------------------------------------------------------------------

\section{\bf Scalar curvature inequalities for contracting maps}\label{S:k-slicings}
This section is devoted to proving Theorem~\ref{T:curvature inequality}.
In Subsection~\ref{SS:k-slicings} we review the results of Schoen and Yau on $k$-slicings of smooth manifolds.
In the remaining part of the section we apply those results to the case of a
nonzero degree map $f\colon X^n\rightarrow S^k\times T^{n-k}$.
In particular, in Subsection~\ref{SS:one-slicings} we consider the case when $k=1$ and deduce Part~(a) of Theorem~\ref{T:curvature inequality}.
The case $k=2$ is studied in Subsection~\ref{SS:two-slicings} and used to prove Part~(b) of Theorem~\ref{T:curvature inequality}.
Finally, in Subsection~\ref{SS:three-slicings} we consider the case $k=3$ and prove Part~(c) of Theorem~\ref{T:curvature inequality}.

%-------------------------------------------------------------------------------------------------------------------------------------------------------------
%-------------------------------------------------------------------------------------------------------------------------------------------------------------
%-------------------------------------------------------------------------------------------------------------------------------------------------------------
%-------------------------------------------------------------------------------------------------------------------------------------------------------------

\subsection{Minimal $k$-slicings: the approach of Schoen and Yau}\label{SS:k-slicings}
In this subsection we recall the notion of minimal $k$-slicing of a smooth manifold and some results about existence and regularity of these objects.

Let $\Sigma_n$ be a closed oriented Riemannian $n$-dimensional smooth
manifold.  
A \emph{minimal $k$-slicing} of $\Sigma_n$ is a nested family of hypersurfaces $\Sigma_k\subset\Sigma_{k+1}\subset\cdots\subset\Sigma_n$, where $\Sigma_j$ is allowed to have singularities for $j<n$, and where, for $j<n$, any $\Sigma_j$ is a minimizer in $\Sigma_{j+1}$ with respect to a suitable weighted volume.
Denote by $\calR_j$ and $\calS_j$ respectively the regular and singular set of $\Sigma_j$.
Notice that $\calR_j$ is open and $\calS_j$ is closed.
For more details, we refer the reader to~\cite{SY17}.

Suppose $F\colon\Sigma_n\rightarrow  Y^k\times T^{n-k}$ is a smooth map of nonzero degree, where $Y^k$ is a closed oriented $k$-dimensional manifold.
Let $\Theta$ denote a $k$-form on $Y$ with $\int_Y\Theta=1$ and let $\theta^{k+1},\ldots,\theta^n$ be the basic normalized one forms on $ T^{n-k}$, i.e. $\int_{S^1}\theta^p=1$, for $p=k+1,\ldots,n$.
We use the notation $\Omega=F^\ast\Theta$ and $\omega^p=F^\ast\theta^p$, for $p=k+1,\ldots,n$.
In the next theorem we collect some existence and regularity results for minimal $k$-slicings in this setting.

%-------------------------------------------------------------------------------------------------------------------------------------------------------------
%-------------------------------------------------------------------------------------------------------------------------------------------------------------

\begin{theorem}[Schoen-Yau,~\cite{SY17}]\label{T:SYslicing}
In the situation described above, there exists a minimal $k$-slicing
$\Sigma_k\subset\cdots\subset\Sigma_n$ such that the following holds.
\begin{enumerate}
	\item\label{item:sing_dim} The Hausdorff dimension of $\calS_j$ is at most $j-3$.
%	\item $\partial\Sigma_j\cap \calR_{j+1}=\emptyset$;
	\item\label{item:curv_est} If $g_n$ is a Riemannian metric on $\Sigma_n$ and $k\leq j\leq n-1$, the inequality
		\begin{equation}\label{E:spectral information}
		4\,\int_{\Sigma_j}|\nabla_j\varphi |^2\,d\mu_j\ \geq\ \int_{\Sigma_j}\big(\scal(g_n)-\scal(g_j)\big)|
		\varphi |^2\,d\mu_j
		\end{equation}
	holds for all functions $\varphi\in C^\infty_c(\calR_j)$, where $g_j$ is the restrtiction of $g_n$ to $\Sigma_j$, 
	$\nabla_j$ is the gradient operator on $(\Sigma_j,g_j)$, and 
	$d\mu_j$ is the $j$-dimensional  Hausdorff measure induced from the
        Riemannian metric on $\Sigma_n$.
	\item\label{item:homology_fix} If $k\leq j\leq n-1$ and $\Sigma_j$ is
          smooth, then $\Sigma_j$ is closed and
		\[
		\int_{\Sigma_j}\Omega\wedge\omega^{k+1}\wedge\cdots\wedge\omega^{j}\ =\ \deg(F)\,.
		\]
\end{enumerate}
Moreover, if $n\leq 7$, then each $\Sigma_j$ is a closed
smooth manifold and the same is true for $n=8$ at least for generic metrics on
$\Sigma_n$ (in the $C^2$-topology).
\end{theorem}

%-------------------------------------------------------------------------------------------------------------------------------------------------------------

\begin{remark}
Parts~\ref{item:sing_dim} and~~\ref{item:homology_fix} correspond
to~\cite[Theorems~2.3 and~2.4]{SY17}.   
% Part~(b) follows from the minimization procedure used in the proof of~\cite[Theorem~4.5]{SY17}.
Finally, Part~\ref{item:curv_est} is obtained from the proof of~\cite[Theorem~2.6]{SY17} as follows.
We observe that~\cite[Inequality~(2.4)]{SY17} holds if we substitute $k=\min_{\Sigma_n}(\scal(g_n))$ with $\scal(g_n)$.
Repeating the remaining part of the proof, we finally obtain the second inequality of~\cite[Theorem~2.6]{SY17} with $k$ replaced by $\scal(g_n)$, from which Inequality~\eqref{E:spectral information} follows.
The last statement follows from classical regularity results for
area-minimizing hypersurfaces: cf.~\cite{Fed70} and \cite{Smale}.
\end{remark}

%-------------------------------------------------------------------------------------------------------------------------------------------------------------
%-------------------------------------------------------------------------------------------------------------------------------------------------------------
%-------------------------------------------------------------------------------------------------------------------------------------------------------------
%-------------------------------------------------------------------------------------------------------------------------------------------------------------

\subsection{One-slicings}\label{SS:one-slicings}
In this subsection we study maps $f\colon X^n\rightarrow T^n$ of nonzero degree.
We use Theorem \ref{T:SYslicing} to prove Part~(a) of Theorem~\ref{T:curvature inequality}.

%-------------------------------------------------------------------------------------------------------------------------------------------------------------
%-------------------------------------------------------------------------------------------------------------------------------------------------------------

\begin{theorem}\label{T:one-slicings}
Let $(X,g)$ be a closed oriented $n$-dimensional Riemannian manifold and let
$f\colon (X,g)\rightarrow \left(T^n,dt_n^2\right)$ be a map of nonzero degree.
Suppose $J\subset S^1$ is a connected open subset such that $f$ is $\epsilon$-contracting on $f^{-1}(J\times T^{n-1})$ and $\scal(g)\geq k_0$ on $f^{-1}(J\times T^{n-1})$ for some constant $k_0>0$.
Then
\begin{equation}\label{E:length contracting}
	\epsilon\ \geq\ \sqrt{k_0}\ \frac{\Length(J)}{2\pi}1\,.
\end{equation}
\end{theorem}

%-------------------------------------------------------------------------------------------------------------------------------------------------------------

\begin{proof}
By part~\ref{item:sing_dim} and \ref{item:homology_fix} of
Theorem~\ref{T:SYslicing}, there exists a one-slicing
$\Sigma_1\subset\cdots\subset \Sigma_n=X$, where $\Sigma_1$ is a closed smooth
one-dimensional manifold and $\mu_1$ is the measure given by the restricted
metric.
Define the map $F_1\colon\Sigma_1\rightarrow S^1$ as the composition
\[
	\Sigma_1\xrightarrow{\ \ i\ \ }X\xrightarrow{\ \ f\ \ }T^n\xrightarrow{\ \ \pi_1\ \ } S^1\,,
\]
where $i\colon \Sigma_1\rightarrow X$ is the inclusion map and where
$\pi_1\colon T^n\rightarrow S^1$ is the projection onto the first $S^1$-factor of $T^n$.
Clearly, $F_1$ is $\epsilon$-contracting on $F_1^{-1}(J)$ as a map of Riemannian manifolds $(\Sigma_1,g_1)\rightarrow (S^1,dt_1^2)$, where $g_1$ is the restriction of $g$ to $\Sigma_1$.
Finally, by parts~\ref{item:curv_est} and~~\ref{item:homology_fix} of Theorem~\ref{T:SYslicing}, $\deg(F_1)=\deg(f)\neq 0$ and the inequality
\begin{equation}\label{E:spectral information on Sigma_1}
	4\,\int_{\Sigma_1}|\nabla_1\varphi |^2\,d\mu_1\ \geq\ \int_{\Sigma_1}\scal(g)|
	\varphi |^2\,d\mu_1
\end{equation}
holds for all functions $\varphi\in C^\infty(\Sigma_1)$. 
Here, we use the fact that, since $\Sigma_1$ is one-dimensional, $\scal(g_1)=0$.

Let $\gamma$ be the closure of a path-component of $F_1^{-1}(J)$, a closed
interval.
Denote by $l$ the length of $\gamma$.
From~\eqref{E:spectral information on Sigma_1} it follows that the first
Dirichlet eigenvalue of the Laplace-Beltrami operator on $\gamma$ is at least
$k_0/4$, as $\scal(g)|_{\gamma}\ge k_0$ by the choice of $J$.
Thus, $k_0/4\leq \pi^2/l^2$ and
\begin{equation}\label{E:length contracting1}
	l\leq \frac{2\pi}{\sqrt{k_0}}\,.
\end{equation}
Since $F_1$ is $\epsilon$-contracting on $f^{-1}(J)$, we also have
\begin{equation}\label{E:length contracting2}
	\Length(J)\,\leq\, \epsilon\, l\,.
\end{equation}
Finally, Inequality~\eqref{E:length contracting} follows from~\eqref{E:length contracting1} and~\eqref{E:length contracting2}.
\end{proof}

%-------------------------------------------------------------------------------------------------------------------------------------------------------------
%-------------------------------------------------------------------------------------------------------------------------------------------------------------
%-------------------------------------------------------------------------------------------------------------------------------------------------------------
%-------------------------------------------------------------------------------------------------------------------------------------------------------------

\subsection{Two-slicings}\label{SS:two-slicings}
In this subsection we use Theorem \ref{T:SYslicing} to study nonzero 
degree maps $f\colon X^n\rightarrow S^2\times T^{n-2}$ and prove Part~(b) of Theorem~\ref{T:curvature inequality}.

%-------------------------------------------------------------------------------------------------------------------------------------------------------------
%-------------------------------------------------------------------------------------------------------------------------------------------------------------

\begin{theorem}\label{T:two-slicings}
Let $(X,g)$ be a closed oriented $n$-dimensional Riemannian manifold and let
$f\colon (X,g)\rightarrow (S^2\times T^{n-2},ds_2^2+dt_{n-2}^2)$ be an area $\epsilon$-contracting map of nonzero degree.
Suppose $\scal(g)\geq k_0$ for some constant $k_0>0$.
Then
\begin{equation}\label{E:codimension-2 inequality}
	\epsilon\ \geq\ \frac{k_0}{4}\,.
\end{equation}
\end{theorem}

%-------------------------------------------------------------------------------------------------------------------------------------------------------------

\begin{proof}
By Theorem~\ref{T:SYslicing}, used as in the proof of
Theorem~\ref{T:one-slicings}, there exist a closed smooth two-dimensional
submanifold $\Sigma_2\subset X$ and a map $F_2\colon (\Sigma_2,g_2)\rightarrow (S^2,ds_2^2)$ such that $\deg(F_2)=\deg(f)\neq 0$, $F_2$ is area $\epsilon$-contracting and the inequality
\begin{equation}\label{E:spectral information on Sigma_2}
	4\,\int_{\Sigma_2}|\nabla_2\varphi |^2\,d\mu_2\ \geq\ \int_{\Sigma_2}\big(\scal(g)-\scal(g_2)\big)|
	\varphi |^2\,d\mu_2
\end{equation}
holds for all functions $\varphi\in C^\infty(\Sigma_2)$.
Here, $g_2$ is the restriction of $g$ to $\Sigma_2$, $\nabla_2$ is the
gradient operator on  $(\Sigma_2,g_2)$, and $\mu_2$ the associated measure.

Let $\Sigma$ be a connected component of $F_2^{-1}(S^2)$.
By~\cite[Theorem~2.7]{SY17}, $\Sigma$ is homeomorphic to the two-sphere.
Moreover, since $F_2$ has nonzero degree and is area $\epsilon$-contracting, we have
\begin{equation}\label{E:area contraction}
	4\pi\ =\ \Area(S^2)\ \leq\ \epsilon\ \Area(\Sigma)\,.
\end{equation}
Using~\eqref{E:area contraction} and~\eqref{E:spectral information on Sigma_2} with the choice $\varphi=1$, we finally deduce
\begin{equation*}
	4\,\pi\ = \ 2\,\pi\chi(\Sigma)\ =\ \int_\Sigma \scal(g_2)\,d\mu_2\ \geq\ \frac{1}{4}\int_\Sigma \scal(g)\,d\mu_2\ \geq\ 
	\frac{k_0}{4}\Area(\Sigma)\ \geq\ \frac{k_0}{\epsilon}\pi\,,
\end{equation*}
from which Inequality~\eqref{E:codimension-2 inequality} follows.
\end{proof}

%-------------------------------------------------------------------------------------------------------------------------------------------------------------
%-------------------------------------------------------------------------------------------------------------------------------------------------------------
%-------------------------------------------------------------------------------------------------------------------------------------------------------------
%-------------------------------------------------------------------------------------------------------------------------------------------------------------

\subsection{Three-slicings}\label{SS:three-slicings}
In this subsection we study maps $f\colon X^n\rightarrow S^3\times T^{n-3}$ of nonzero degree.
Using the minimal hypersurfaces technique, we construct a three-dimensional
manifold $X^3$ with a metric of positive scalar curvature and a nonzero degree
map $f_3\colon X^3\rightarrow S^3$.
Since every orientable three-dimensional manifold admits a spin structure, we
use the spin-Dirac operator technique to prove Part~(c) of
Theorem~\ref{T:curvature inequality}. 

%-------------------------------------------------------------------------------------------------------------------------------------------------------------
%-------------------------------------------------------------------------------------------------------------------------------------------------------------

\begin{theorem}\label{T:codimension-3}
Let $(X,g)$ be a closed oriented Riemannian manifold with $n=\dim X\leq 8$ and
let $f\colon (X,g)\rightarrow (S^3\times T^{n-3},ds_3^2+dt_{n-3}^2)$ be an area $\epsilon$-contracting map of nonzero  degree.
If $\scal(g)\geq k_0$ for some constant $k_0>0$, then
\begin{equation}\label{E:area-contracting inequality}
	\epsilon\ \geq\ \frac{k_0}{6}\,.
\end{equation}
\end{theorem}

%-------------------------------------------------------------------------------------------------------------------------------------------------------------
%-------------------------------------------------------------------------------------------------------------------------------------------------------------

Before proving the theorem, we recall some facts about the scalar curvature under a conformal change of the metric.
Let $(M,g)$ be an $n$-dimensional closed Riemannian manifold.
The conformal Laplacian on $(M,g)$ is the operator
\begin{equation}
	L_g\,:=\,-c(n)\,\Delta_g\,+\,\scal(g)\,,
\end{equation}
where $\Delta_g$ is the Laplace-Beltrami operator on $(M,g)$, and where $c(n)=4(n-1)/(n-2)$.
Let $\phi$ be a positive smooth function on $M$ and consider the metric $\widetilde{g}=\phi^{\frac{4}{n-2}}g$.
The scalar curvature of $\widetilde{g}$ is given by the formula
\begin{equation}\label{E:conformal change}
	\scal(\widetilde{g})\,=\,\phi^{-\frac{n+2}{n-2}}\,L_g(\phi)\,.
\end{equation}
In the proof of Theorem~\ref{T:codimension-3} we also make use of functions with a local area contracting factor in the following precise sense.
Let $f\colon (X,g)\rightarrow (Y,g_Y)$ be a differentiable map between smooth
Riemannian manifolds and let $\psi\colon X\rightarrow\RR_+$ be a smooth map.
We say that $f$ is \emph{area $\psi$-contracting} if $\|(f_\ast v)\wedge (f_\ast w)\|_{f(x)}\leq\psi(x)\|v\wedge w\|_x$, 
for all $x\in X$ and all $v,w\in T_xX$.

%-------------------------------------------------------------------------------------------------------------------------------------------------------------
%-------------------------------------------------------------------------------------------------------------------------------------------------------------

\subsection{Proof of Theorem~\ref{T:codimension-3}}
By Theorem~\ref{T:SYslicing}, used as in the proof of
Theorem~\ref{T:one-slicings}, there exist a closed smooth three-dimensional
submanifold $\Sigma_3\subset X$ and a map $F_3\colon (\Sigma_3,g_3)\rightarrow (S^3,ds_3^2)$ such that $\deg(F_3)=\deg(f)\neq 0$, $F_3$ is area $\epsilon$-contracting and the inequality
\begin{equation}\label{E:spectral information on Sigma_3}
	4\,\int_{\Sigma_3}|\nabla_3\varphi |^2\,d\mu_3\ \geq\ \int_{\Sigma_3}\big(\scal(g)-\scal(g_3)\big)\,
	\varphi^2\,d\mu_3
\end{equation}
holds for all functions $\varphi\in C^\infty(\Sigma_3)$. Strictly speaking, if
$n=8$ we might have to change $g$, but the curvature bound will be essentially
unchanged and we ignore this detail from now on.
Here, $g_3$ is the metric $g$ restricted to $\Sigma_3$, $\nabla_3$ is the
gradient operator on  $(\Sigma_3,g_3)$, and $\mu_3$ the measure induced by
$g_3$.
Let $L_3$ be the conformal Laplacian of $(\Sigma_3,g_3)$, i.e. $L_3=-8\Delta_3+\scal(g_3)$.
Since $\scal(g)\geq k_0$, from~\eqref{E:spectral information on Sigma_3} we deduce
\begin{equation}\label{E:spectral information on L_3}
	\int_{\Sigma_3}\varphi\, L_3(\varphi)\,d\mu_3\ \geq\ k_0\,\int_{\Sigma_3}\varphi^2\,d\mu_3\,,
	\qquad \forall\varphi\in C^\infty(\Sigma_3)\,.
\end{equation}
We now use the method of Schoen and Yau~\cite{SY79} to construct a metric of positive scalar curvature on $\Sigma_3$.
Let $\lambda_1$ be the first eigenvalue of the operator $L_3$.
By the usual variational characterization, Inequality~\eqref{E:spectral information on L_3} implies that $\lambda_1\geq k_0$.
Let $\phi\in C^\infty(\Sigma_3)$ be an eigenfunction relative to $\lambda_1$.
It is well-known that $\phi$ doesn't vanish at any point so we assume that $\phi>0$ and define the metric $\widetilde{g}_3=\phi^4g_3$.
By~\eqref{E:conformal change}, its scalar curvature satisfies
\begin{equation}\label{E:conformal metric on Sigma_3}
	\scal\left(\widetilde{g}_3\right)\,=\,\phi^{-5}L_3(\phi)\,=\,\phi^{-4}\lambda_1
	\,\geq\,\phi^{-4}\,k_0\,.
\end{equation}
Moreover, the map $F_3$ is area $(\epsilon\phi^{-4})$-contracting with respect to the metric $\widetilde{g}_3$.

Fix a number $R>0$ and let $S^1_R$ be the circle of radius $R$ endowed with the standard metric $dt_R^2$.
Let $\sigma\colon S^3\times S^1\rightarrow S^3\wedge S^1\cong S^4$ be a $1$-contracting ``smashing'' map of nonzero degree.
Define the nonzero degree map $f_4\colon \Sigma_3\times S^1_R\rightarrow S^4$ through the composition
\[
	\Sigma_3\times S^1_R\xrightarrow{\ F_3\times \frac{1}{R}\ }S^3\times S^1\xrightarrow{\ \ \sigma\ \ }
	S^4\,.
\]
Endow $\Sigma_3\times S^1_R$ with the product metric $\widetilde{g}_4:=\widetilde{g}_3+dt_R^2$.
Extend $\phi$ to a function on $\Sigma_3\times S^1_R$ by making it constant in
the $S^1_R$-direction.
With a slight abuse of notation, denote this function also by $\phi$.
Notice that with respect to $\widetilde{g}_4$ the map $f_4$ is area $(\epsilon\phi^{-4})$-contracting on vectors tangent to $\Sigma_3$ and $(1/R)$-contracting on vectors tangent to $S^1_R$.
Notice also that, since the metric $dt_R^2$ is flat, the scalar curvature of $\widetilde{g}_4$ satisfies
\begin{equation}\label{E:inequality S_4}
	\scal\left(\widetilde{g}_4\right)\ \geq\ \phi^{-4}\,k_0 >0.
\end{equation}

In order to obtain Inequality~\eqref{E:area-contracting inequality}, we use the method of Gromov and Lawson~\cite{GL83}.
Since $\Sigma_3$ has dimension $3$, it is a spin manifold and $\Sigma_3\times S^1_R$ is spin as well.
Choose a spin structure on $\Sigma_3\times S^1_R$ and let $\slashed{S}$ be the associated complex spinor bundle with Dirac operator $\slashed{D}$.
We now fix a vector bundle with connection on $S^4$ following  Llarull~\cite{Ll98}.
Set
\[
	E_0\,:=\,P_{Spin_4}(S^4)\times_\lambda \C l_4\,,
\]
with the metric and connection $\nabla^{E_0}$ induced from $(S^4,ds_4^2)$.
Here, $P_{Spin_4}(S^4)$ is the principal $\Spin_4$-bundle defining the spin structure on $TS^4$ and $\lambda$ is the representation given by left multiplication.
Using the map $f_4$, pull-back the bundle $E_0$ to $\Sigma_3\times S^1_R$ together with its connection.
Doing so, we obtain a bundle $E=f_4^\ast E_0$ with connection $\nabla=f_4^\ast\nabla^{E_0}$.
Let $\slashed{D}_E\colon \Gamma(\slashed{S}\tensor E)\rightarrow \Gamma(\slashed{S}\tensor E)$ be the operator $\slashed{D}$ twisted with the bundle $E$.

The volume element on $S^4$ gives a grading $E_0=E_0^+\oplus E_0^-$.
Set $\slashed{D}_{E^+}:=\slashed{D}_E|_{\Gamma(\slashed{S}\tensor E^+)}$.
It is an essentially self-adjoint elliptic operator of order one acting on smooth sections of the bundle $\slashed{S}\tensor E^+$.
By classical results on elliptic operators, its kernel is a finite dimensional vector space.
Since $\Sigma_3\times S^1_R$ is even-dimensional, we have the splitting $\slashed{S}=\slashed{S}^+\oplus\slashed{S}^-$.
This induces a $\ZZ_2$-grading $\Gamma(\slashed{S}\tensor E^+)=\Gamma(\slashed{S}^+\tensor E^+)\oplus\Gamma(\slashed{S}^-\tensor E^+)$ and the operator $\slashed{D}_{E^+}$ is odd with respect to this grading.
The index of $\slashed{D}_{E^+}$ is defined as the integer 
\[
	\ind(\slashed{D}_{E^+})\ :=\ \dim\Ker\slashed{D}_{E^+}^+\,-\,\dim\Ker\slashed{D}_{E^+}^-\in\ZZ\,,
\]
where $\slashed{D}_{E^+}^\pm:=\slashed{D}_E|_{\Gamma(\slashed{S}^\pm\tensor E^+)}$.
Since $c_2(E_0^+)\neq 0$ (see~\cite[page~66]{Ll98}), $\deg(f_4)\neq 0$, and
$p_1(\Sigma_3\times S^1_R)=0$ so that the total $\hat A$-genus of
$\Sigma_3\times S^1_R$ equals $1\in H^0(\Sigma^3\times S^1_R;\rationals)$, the
Atiyah-Singer index theorem used as 
in~\cite{GL83} implies that $\ind(\slashed{D}_{E^+}) \neq 0$.

In order to conclude the proof, we study the kernel of the operator $\slashed{D}_E^2$.
Pick a section $u\in \Gamma(\slashed{S}\tensor E)$.
The Bochner-Lichnerowicz-Weitzenb\"ock-Schr\"odinger formula (see~\cite[Theorem~8.17]{LM89}) implies
\begin{equation}\label{E:BLW1}
	\left<\slashed{D}^2_Eu,u\right>\ \geq\ \frac{1}{4}\left<\scal\left(\widetilde{g}_4\right)u,u\right>\,+\,
	\left<\calR^Eu,u\right>\,,
\end{equation}
where $\calR^E\in\End(\slashed{S}\tensor E)$ depends linearly on the components of the curvature of $\nabla^E$.
Here, $\left<\cdot,\cdot\right>$ denotes the inner product
\[
	\left<v,w\right>\ :=\ \int_{\Sigma_3\times S^1_R}\left<v,w\right>_x\,,
	\qquad\qquad v,w\in\Gamma(\slashed{S}\tensor E)\,,
\]
where $\left<\cdot,\cdot\right>_x$ is the inner product of the fiber $\slashed{S}_x\tensor E_x$.
We now estimate separately the two terms on the right hand side of ~\eqref{E:BLW1}.
From~\eqref{E:inequality S_4}, we have
\begin{equation}\label{E:BLW2}
	\left<\scal\left(\widetilde{g}_4\right)u,u\right>\ \geq\ k_0\,\left\|\phi^{-2}u\right\|^2\,.
\end{equation}
In order to estimate the second term we closely follow~\cite{Ll98}.
Fix a point $x\in \Sigma_3\times S^1_R$ and let $\{e_1,e_2,e_3,e_4\}$ be a $\widetilde{g}_4$-orthonormal basis of tangent vectors around $x$ such that $(\nabla e_j)_x=0$, $\{e_1,e_2,e_3\}$ are tangent to $\Sigma_3$ and $e_4$ is tangent to $S^1_R$.
Also choose an orthonormal basis $\{\epsilon_1,\epsilon_2,\epsilon_3,\epsilon_4\}$ of tangent vectors around $f(x)$ in such a way that $(\nabla \epsilon_j)_{f(x)}=0$ and $\epsilon_j=\lambda_j f_\ast e_j$ for suitable positive scalars $\{\lambda_j\}_{j=1}^4$.
The conditions $(\nabla e_j)_x=0$ and $(\nabla \epsilon_j)_{f(x)}=0$ are needed in order to use results from~\cite[Section~4]{Ll98}.
For $1\leq i,j\leq 3$ and $i\neq j$, we have
\[
	1\,=\,\|\epsilon_i\wedge\epsilon_j\|_x\,=\,\lambda_i\lambda_j\|f_\ast (e_i)\wedge f_\ast (e_j)\|_x\,\leq\,
	\lambda_i\lambda_j\epsilon\phi^{-4}\|e_i\wedge e_j\|_x\,=\,\lambda_i\lambda_j\epsilon\phi^{-4}\,,
\]
from which
\begin{equation}\label{E:contracting factor1}
	\frac{1}{\lambda_i\lambda_j}\,\leq\,\epsilon\phi^{-4}\,,\qquad\qquad 1\leq i,j\leq 3\,,\,\,i\neq j\,.
\end{equation}
Since $\Sigma_3$ is compact, $f_3$ is $c$-contracting for some constant $c$.
Therefore, for $i=1,2,3$ we have
\[
	1\,=\,\|\epsilon_i\wedge\epsilon_4\|_x\,=\,\lambda_i\lambda_4\|f_\ast (e_i)\wedge f_\ast(e_4)\|_x\,\leq\,
	\lambda_i\lambda_4\frac{c}{R}\,,
\]
from which
\begin{equation}\label{E:contracting factor2}
	\frac{1}{\lambda_i\lambda_4}\,\leq\,\frac{c}{R}\,,\qquad\qquad 1\leq i\leq 3\,.
\end{equation}
Using Inequalities~\eqref{E:contracting factor1} and~\eqref{E:contracting factor2},~\cite[Formula~(4.4)]{Ll98}, and~\cite[Lemma~4.5]{Ll98}, we get
\begin{align*}
	\left<\calR^Eu,u\right>_x &\geq\ -\frac{1}{4}\sum_{ i\neq j}\frac{1}{\lambda_i\lambda_j}\|u\|^2_x\\
	&= \ -\frac{1}{4}\sum_{1\leq i\neq j\leq 3}\frac{1}{\lambda_i\lambda_j}\|u\|^2_x
		\,-\frac{1}{2}\,\sum_{i=1}^3\frac{1}{\lambda_i\lambda_4}\|u\|^2_x\\
	&\geq\ -\frac{1}{4}\left\{6\,\epsilon\phi^{-4}\left\|u\right\|^2_x
		\,+\,6\,\frac{c}{R}\|u\|^2_x\right\}\,.
\end{align*}
Integrating, we obtain
\begin{equation*}\label{E:BLW3}
	\left<\calR^Eu,u\right>\ \geq\ -\frac{1}{4}\left\{6\,\epsilon\left\|\phi^{-2}u\right\|^2
		\,+\,6\,\frac{c}{R}\|u\|^2\right\}\,.
\end{equation*}
This inequality together with~\eqref{E:BLW1} and~\eqref{E:BLW2} implies
\begin{equation}\label{E:BLW4}
	\left<\slashed{D}^2_Eu,u\right>\ \geq\ \frac{1}{4} \left\{(k_0\,-\,6\epsilon)\,\left\|\phi^{-2}u\right\|^2\,
	-6\,\frac{c}{R}\|u\|^2\right\}\,.
\end{equation}
Finally, notice that $\Ker \slashed{D}_{E^+}\subset \Ker \slashed{D}_E\subset \Ker \slashed{D}_E^2$.
Therefore, since $\phi>0$, $\ind(\slashed{D}_{E^+})\neq 0$ and $R$ can be chosen arbitrarily large, Inequality~\eqref{E:BLW4} implies Inequality~\eqref{E:area-contracting inequality}.
\hfill$\square$

%-------------------------------------------------------------------------------------------------------------------------------------------------------------
%-------------------------------------------------------------------------------------------------------------------------------------------------------------
%-------------------------------------------------------------------------------------------------------------------------------------------------------------
%-------------------------------------------------------------------------------------------------------------------------------------------------------------
%-------------------------------------------------------------------------------------------------------------------------------------------------------------
%-------------------------------------------------------------------------------------------------------------------------------------------------------------
%-------------------------------------------------------------------------------------------------------------------------------------------------------------
%-------------------------------------------------------------------------------------------------------------------------------------------------------------

\section{\bf A ``cut-and-paste" construction}\label{S:cut-and-paste}
This section is devoted to proving Theorem~\ref{T:enlargeable nonspin}.
The existence and regularity result \ref{T:SYslicing} of Schoen and Yau
requires a map from a closed manifold into a torus.
The notion of enlargeability for a metric is defined by means of maps from (possibly noncompact) manifolds into spheres.
The next theorem allows us to use Theorem \ref{T:SYslicing} in the latter
setting. It is proved by using a ``cut-and-paste'' construction.

%-------------------------------------------------------------------------------------------------------------------------------------------------------------
%-------------------------------------------------------------------------------------------------------------------------------------------------------------

\begin{theorem}\label{T:Lego}
Let $(X,g)$ be an oriented $n$-dimensional Riemannian manifold without boundary.
Suppose $f\colon (X,g)\rightarrow (S^n,ds_n^2)$ is an $\epsilon$-contracting map which is constant at infinity and of nonzero degree.
Then there exist a closed oriented $n$-dimensional Riemannian manifold
$(\Sigma_n,g_n)$, a map $F\colon (\Sigma_n,g_n)\rightarrow \big(T
^n, dt_n^2\big)$, a connected open set $J\subset S^1$ and a constant $c_n$ such that
\begin{enumerate}
	\item $\deg(F)=\deg(f)$;
	\item the map $F$ is $(c_n\epsilon)$-contracting on $F^{-1}(J\times  T^{n-1})$;
	\item the set $F^{-1}(J\times  T^{n-1})$ is isometric to an open subset of $X$.
\end{enumerate}
Moreover, the set $J$ and the constant $c_n$ are independent of the Riemannian manifold $(X,g)$ and the map $F$.
\end{theorem}

%-------------------------------------------------------------------------------------------------------------------------------------------------------------
%-------------------------------------------------------------------------------------------------------------------------------------------------------------

\begin{proof}
By hypothesis, there exists a point $p\in S^n$ such that $f^{-1}(S^n\setminus\{p\})$ is relatively compact.
Embed the torus $ T^{n-1}$ into $S^n$ with trivial normal bundle and in such a way that $p\in S^n\setminus  T^{n-1}$.
Let $V^\prime$ be a closed tubular neighborhood of $ T^{n-1}$ in $S^n$ such that $p\in S^n\setminus V^\prime$.
By a close $C^1$-approximation, make $f$ transversal to $\partial V^\prime=S^0\times  T^{n-1}$ and define the compact submanifold $\Sigma_n^\prime :=f^{-1}(D^1\times  T^{n-1})$ of $X$ with boundary $\partial \Sigma_n^\prime=f^{-1}(\partial V^\prime)$.

Let $V$ be the oriented closed manifold obtained as the double of $V^\prime$.
Since $ T^{n-1}$ has trivial normal bundle in $S^n$, there is a diffeomorphism
$\Phi\colon V\rightarrow S^1\times  T^{n-1}$.
Clearly, $\deg(\Phi)=1$.
We endow $ T^n=S^1\times  T^{n-1}$ with the standard product metric $ dt_n^2$.
Pick a connected open interval $J\subset D^1$ and let $g_V$ be a metric on $V$ coinciding with the round metric of $S^n$ on $\Phi^{-1}(J\times  T^{n-1})$.
Since the manifold $V$ is compact, the map $\Phi\colon (V,g_V)\rightarrow ( T^n, dt_n^2)$ is $c$-contracting for some constant $c$.

We now define the closed oriented manifold $\Sigma_n$ as the double of $\Sigma_n^\prime$.
The map $f$ induces a map $Q\colon \Sigma_n\longrightarrow V$.
Since a regular value $x$ of $Q$ can be regarded as a regular value of $f$ and
$Q^{-1}(x)=f^{-1}(x)$, we have $\deg(Q)=\deg(f)$.

Define on $\Sigma_n$ a metric $g_n$ coinciding with $g$ on the open set $f^{-1}(J\times  T^{n-1})$.
Let $F\colon (\Sigma_n,g_n)\rightarrow \big( T^n,\, dt_n^2\big)$ be the map defined through the composition
\[
	(\Sigma_n,g_n)\xrightarrow{\ \ Q\ \ } (V,g_V)\xrightarrow{\ \ \Phi\ \ } 
	\Big( T^n,\, dt_n^2\Big)\,.
\]
Clearly, $\deg(F)=\deg(f)$.
Notice that we can view $J\times  T^{n-1}$ as a subset of both, $V^\prime\subset S^n$ and $ T^n$, and that $F^{-1}(J\times  T^{n-1})$ coincides with $f^{-1}(J\times  T^{n-1})$.
Therefore, $F^{-1}(J\times  T^{n-1})$ is isometric to an open subset of $X$ and $F$ is $(c_n\epsilon)$-contracting when restricted to $F^{-1}(J\times  T^{n-1})$.
Finally, notice that the set $J$ and the constant $c_n$ depend only on the ``cut-and-paste" manipulations of the target space $S^n$.
\end{proof}

From Theorem~\ref{T:Lego} and Theorem~\ref{T:one-slicings} we directly deduce the following consequence.

%-------------------------------------------------------------------------------------------------------------------------------------------------------------
%-------------------------------------------------------------------------------------------------------------------------------------------------------------

\begin{corollary}\label{C:length contracting inequality}
Let $(X,g)$ be a connected oriented $n$-dimensional Riemannian manifold without boundary.
Suppose $f\colon (X,g)\rightarrow (S^n,ds_n^2)$ is a map which is constant at infinity, of nonzero  degree and $\epsilon$-contracting. 
If $\scal(g)\geq k_0>0$, then there exists a constant $b_n$, depending only on the dimension of the manifold $X$, such that
\[
	\epsilon\ \geq\ b_n\sqrt{k_0}\,.
\]
\end{corollary}

We now use the inequality of the previous corollary to prove Theorem~\ref{T:enlargeable nonspin}.

%-------------------------------------------------------------------------------------------------------------------------------------------------------------
%-------------------------------------------------------------------------------------------------------------------------------------------------------------
%-------------------------------------------------------------------------------------------------------------------------------------------------------------
%-------------------------------------------------------------------------------------------------------------------------------------------------------------

\subsection{Proof of Theorem~\ref{T:enlargeable nonspin}}
Let $M$ be a connected oriented manifold without boundary of dimension $n$ and let $g$ be a Riemannian metric on $M$ such that $\scal(g)\geq k_0$ for some constant $ k_0>0$.
Suppose $\bar M\rightarrow M$ is a connected oriented cover and $f\colon (\bar M,\bar g)\rightarrow (S^n, ds_n^2)$ is a map which is constant at infinity, of nonzero degree, and $\epsilon$-contracting.
Here, $\bar g$ denotes the metric $g$ lifted to $\bar M$.
By Corollary~\ref{C:length contracting inequality}, there exists a constant $b_n$, depending only on $n$, such that $\epsilon\ \geq\ b_n\sqrt{k_0}$.
Therefore, $g$ cannot be enlargeable.
\hfill$\square$

%-------------------------------------------------------------------------------------------------------------------------------------------------------------
%-------------------------------------------------------------------------------------------------------------------------------------------------------------
%-------------------------------------------------------------------------------------------------------------------------------------------------------------
%-------------------------------------------------------------------------------------------------------------------------------------------------------------
%-------------------------------------------------------------------------------------------------------------------------------------------------------------
%-------------------------------------------------------------------------------------------------------------------------------------------------------------
%-------------------------------------------------------------------------------------------------------------------------------------------------------------
%-------------------------------------------------------------------------------------------------------------------------------------------------------------

\section{\bf Nonenlargeable metrics on noncompact manifolds}\label{S:nonenlargeable metrics}
This section is devoted to the proof of Theorem~\ref{T:nonenlargeable metrics}.
To this end, we have to construct appropriate metrics which violate the
enlargeability condition.
The following lemma describes such metrics.

%-------------------------------------------------------------------------------------------------------------------------------------------------------------
%-------------------------------------------------------------------------------------------------------------------------------------------------------------

\begin{lemma}\label{lem:non-extendable_gives_nonenlargeable}
  Assume that $M$ is a manifold with a Riemannian metric $g$ such that there
  is $C>0$ and for
  each point 
  $x\in M$ there is a $1$-Lipschitz curve $\gamma\colon [0,c)\to M$ with $c\le
  C$ and with $\gamma(0)=x$
  which has no continuous extension to $[0,c]$.
  Then $M$ is not enlargeable.
\end{lemma}

%-------------------------------------------------------------------------------------------------------------------------------------------------------------

\begin{proof}
  Note that every covering $\bar M\to M$ has the same property, since the required
  non-extendable curves can be obtained as lifts of the non-extendable curves
  in $M$.

  Observe also that, if $K\in\bar M$ is compact and $\gamma\colon [0,c)\to
  \bar M$ is non-extendable $1$-Lipschits, then the image of $\gamma$ has to
  leave $K$. Otherwise, there would be a limit point of the sequence
  $\left(\gamma(c-1/k)\right)_{k\in\NN}$ in
  $K$. On the other hand, due to the $1$-Lipschitz property,
  $\left(\gamma(c-1/k)\right)_{k\in\NN}$ is a Cauchy sequence so that we could
  extend $\gamma$ continuously to $c$.

  Assume now that $\epsilon<\frac{1}{C}$ and $f\colon \bar M\to S^n$ is a
  compactly 
  supported $\epsilon$-contracting map (say supported on $K
 \subset \bar M$). For an arbitrary $x\in \bar M$ pick a $1$-Lipschitz curve
 $\gamma\colon [0,c)\to \bar M$ with $\gamma(0)=x$ which cannot be extended
 over $c<C$. As $\gamma$ leaves $K$, it follows that some points on the curve
 $\gamma$ are mapped to the base point in $S^n$. Because of the $1$-Lipschitz
 property and because $f$ is $\epsilon$-contracting, every point on the curve
 $\gamma$ is mapped to a point of distance $\le \epsilon c<\epsilon
 C<1$ from the base point. This applies in particular to the point $x\in\bar
 M$, which was arbitrary. Now the diameter of $S^n$ is larger than $1$,
 therefore $f$ is not surjective and consequently $\deg(f)=0$.
\end{proof}

%-------------------------------------------------------------------------------------------------------------------------------------------------------------

To apply this to prove Theorem~\ref{T:nonenlargeable metrics}, we need
conditions which imply the existence of the non-extendable paths. 
The following lemma provides a basic such condition.

%-------------------------------------------------------------------------------------------------------------------------------------------------------------
%-------------------------------------------------------------------------------------------------------------------------------------------------------------

\begin{lemma}\label{lem:non-complete_gives_non-extendable}
  Assume that $M$ is a connected Riemannian manifold of diameter less than
  $C>0$ and 
  that $M$ is not complete as a metric space. Then for each $x\in M$ there is
  a $1$-Lipschitz path $\gamma\colon [0,c)\to M$ with $c<C+2$ and with
  $\gamma(0)=x$ which cannot be extended continuously to $c$.
\end{lemma}

%-------------------------------------------------------------------------------------------------------------------------------------------------------------

\begin{proof}
  Choose a Cauchy sequence $(x_k)_{k\in\NN}$ in $M$ which does not converge,
  which exists due to the fact that $M$ is not complete. By passing to a
  subsequence, we can assume that $d(x_k,x_{k+1})\le 2^{-k-1}$ for each $k$.

  Choose a $1$-Lipschitz path of length $<C+1$ from $x$ to $x_1$ and
  $1$-Lipschitz paths from $x_k$ to $x_{k+1}$ of length $<2^{-k}$. These paths
  exist due to the fact that $M$ is a path-metric space as almost geodesics
  parametrized by arc length. Their concatenation is a path $\gamma\colon
  [0,c)\to M$ which satisfies all the conditions. It can not be extended
  continuously to the closed interval, because the value at the endpoint would
  have to be a limit point of the non-convergent Cauchy sequence
  $(x_k)_{k\in\NN}$.
\end{proof}

%-------------------------------------------------------------------------------------------------------------------------------------------------------------

Now we can put together the observations made so far to prove Theorem~\ref{T:nonenlargeable metrics}. 

\subsection{Proof of Theorem~\ref{T:nonenlargeable metrics}}
Using \cite[Theorem 2]{NO61}, we choose a
finite diameter Riemannian metric $g$ on 
$M$. The Heine-Borel property says that a complete finite diameter Riemannian
manifold is compact. As $M$ by assumption is not compact, it is not complete
for $g$. By
Lemma \ref{lem:non-complete_gives_non-extendable} there are the required
non-extendable paths in $M$ to apply Lemma
\ref{lem:non-extendable_gives_nonenlargeable} and the thesis follows.
\hfill$\square$

%-------------------------------------------------------------------------------------------------------------------------------------------------------------
%-------------------------------------------------------------------------------------------------------------------------------------------------------------
%-------------------------------------------------------------------------------------------------------------------------------------------------------------
%-------------------------------------------------------------------------------------------------------------------------------------------------------------
%-------------------------------------------------------------------------------------------------------------------------------------------------------------
%-------------------------------------------------------------------------------------------------------------------------------------------------------------
%-------------------------------------------------------------------------------------------------------------------------------------------------------------
%-------------------------------------------------------------------------------------------------------------------------------------------------------------

\end{document}